\documentclass[10pt]{article}
\usepackage{setspace}
\usepackage{graphicx}
\usepackage{amsthm}
\usepackage{subcaption}
\usepackage{amsmath,amssymb}
\usepackage{multirow}

\usepackage[colorlinks=true,linkcolor=black]{hyperref}
\hypersetup{citecolor=black}

\usepackage[a4paper,bindingoffset=0.5in,
left=.5in,right=0.5in,
top=1in,bottom=1in,footskip=.25in]{geometry}

\newtheorem{theorem}{Theorem}[section]

\newtheorem{lemma}[theorem]{Lemma}

\title{ An improved Haar wavelet quasilinearization technique for a class of generalized Burger's equation }
\author{Amit Kumar Verma$^a$, Mukesh Kumar Rawani$^b$\thanks{Email:$^a$akverma@iitp.ac.in,$^b$mukesh.ism1990@gmail.com}
	\\\small{\textit{$^{a,b}$Department of Mathematics,}}\\\small{\textit{Indian Institute of Technology Patna, Patna--$801106$, Bihar, India.}}
}
\date{\today}
\begin{document}
\maketitle
\begin{abstract}
	Solving Burgers' equation always poses challenge to researchers as for small values of viscosity the analytical solution breaks down. Here we propose to compute  numerical solution for a class of  generalised Burgers' equation described as $$ \frac{\partial w}{\partial t}+ w^{\mu}\frac{\partial w}{\partial x_{*}}=\nu w^{\delta}\frac{\partial^2 w}{\partial x^{2}_{*}},\hspace{0.2in}a \leq x_{*}\leq b,~~t\geq 0,$$ based on the Haar wavelet (HW) coupled with quasilinearization approach. In the  process of numerical solution, finite forward difference is applied to discretize the time derivative, Haar wavelet to spatial derivative  and non-linear term is linearized by quasilinearization technique. To discuss the accuracy and efficiency of the method $L_{\infty}$ and $L_{2}$-error norm are computed and they are compared with some existing results. We have proved the convergence of the proposed method. Computer simulations show that the present method gives accurate and  better result even for small number of grid points for small values of viscosity.
\end{abstract}
\section{Introduction}
The study of Burger's equation \cite{bateman1915some,burgers1939mathematical,burgers1948mathematical} is popular among the scientific community. Different form of Burgers' equation appears in verious  areas such as  chemical kinetics, optical fibers, fluid dynamics, biology, solid state physics, plasma physics etc. The main aim of this article is to find numerical solution of a class of generalized Burgers' equation of the following form
\begin{equation}
\label{eq1}
\frac{\partial w}{\partial t}+ w^{\mu}\frac{\partial w}{\partial x_{*}}=\nu w^{\delta}\frac{\partial^2 w}{\partial x^{2}_{*}},\hspace{0.2in}a \leq x_{*}\leq b,~~t\geq 0,
\end{equation}
with the boundary conditions (BCs)
\begin{equation}
w(a,t)=f_{1}(t) \hspace{0.5in} w(b,t)=f_{2}(t),\hspace{0.5 in} t\in [0,T],
\end{equation}
and the initial condition
\begin{equation}
w(x_{*},0)=f(x_{*}), \hspace{0.5in} x_{*}\in[a,b],
\end{equation}
where $w(x_{*},t)$ represents the velocity for the spacial dimension $x_{*}$, time $t$, and  $\nu > 0$ is kinematic coefficient and  $\mu,\delta$ are non negative integer such that $\mu+\delta\geq 1$. When $\mu=1,\delta=0$, equation \eqref{eq1} is called Burgers' equation. In $1915$, it  was first proposed by Bateman \cite{bateman1915some}. Later in 1948, it was introduced by Burger \cite{burgers1939mathematical,burgers1948mathematical}  as a class of equation which delineate the mathematical model of turbulence.  Due to his immense works on the model it is termed as Burgers' equation. For arbitrary initial condition, it was solved analytically by both Hopf \cite{hopf1950partial} and Cole \cite{cole1951quasi} independently. Since these analytical solution are in infinite series and converge very slowly for small value of viscosity coefficient $\nu$. In many cases these solution fails for $\nu <0.01$ and not easy to capture the solution.\par
For $\mu+\delta\geq 2$  Eq. \eqref{eq1} is called generalized Burger's equation. The generalized Burgers' equation has been solved numerically and analytically by several researcher.
 For $\mu\geq 2,\delta=0$, Ramadan EL-Danaf \cite{ramadan2005numerical} used collocation method coupled with quintic  splines, Ramdan et al. \cite{ramadan12005numerical} discussed collocation of septic over finite element to find numerical solution. Saka and Dag \cite{saka2008numerical} used time and space splitting techniques and then applied quintic B-spline collocation method. Petrov-Galerkin technique in \cite{roshan2011numerical} applied by Roshan and Bhamra. Brastos \cite{bratsos2010fourth}-\cite{bratsos2011explicit} used various  explicit finite difference scheme to find the numerical solution. Zhang et al. \cite{rong2013modified} have used local discontinuous Galerkin method. Several others methods have been developed to solve the equation discussed in \cite{temsah2009numerical,griewank2009efficient,duan2008lattice} etc.\par
Wavelets method has been used to solve PDEs (Partial differential equations) numerically since $1990$s. The best feature of the wavelet approch is the capability to detect the irregular structure, singularities and transient phenomena revealed by the analyzed equation. Most of the algorithm  for the numerical solution of PDEs by  the wavelet method are  based on collocation \cite{bertoluzza1994wavelet}-\cite{comincioli2000wavelet} or  the Galerkin technique method \cite{bertoluzza1994wavelet,chen1996computation,avudainayagam1999wavelet}.
Based on the Haar wavelets method, Chen and Hsiao in \cite{chen1997haar} proposed a method to find the numerical solution of ordinary differential equations. They replaced highest order derivative function by Haar series.  Recently,   many ordinary and partial differential equation have been solved by several authors by Haar wavelet method. Lepik used Haar wavelet method to solve nonlinear ordinary differntial equation and diffusion equation in \cite{lepik2005numerical}, Poisson equation in \cite{lepik2011solving}, Sin-Gordon and Burgers' equation in \cite{lepik2007numerical}. Verma et al. solved Lane-Emden Equations in \cite{verma2019higher}. Celik \cite{ccelik2012haar} discussed Haar wavelet method for the numerical solutions of Burger-Huxley equation and later applied  to magneto hydrodynamic flow equation. Jiwari \cite{jiwari2012haar} used Haar wavelet quasilinearization  approach to solve Burgers' equation.  Lane-Emden equation arising in astrophysics has been solved numerically using Haar wavelet method in \cite{kaur2013haar}. Haar wavelet method  is used in  \cite{kaur2013haar} for the numerical solution of biharmonic and $2D$ and $3D$ Poisson equation. 

In this paper, we propose a technique to solve a class of generalized Burgers' equation with the combination of finite forward  difference and Haar wavelet. We discretize the time derivative $w_t$ by forward finite difference. Also the other terms $w^{\mu}w_{x}$ , $w^{\delta}w_{xx}$ are approximated by their average at $j$ and $j+1^{th}$ level. Numerical simulations suggest that this averaging improves the results. The spatial discretization is taken care by Haar wavelet and quasilinearization technique is used to deal with the non-linear term. Convergence analysis is also presented by computing $L^2$ error estimate analytically. Test problems are considered at the end of the section to validate the robustness of the present method. \par 

This article is organised as follows. In section \ref{sec2}, introduction about the Haar wavelets and approximation of the function by using Haar wavelet is discussed. In section \ref{sec4} discretization of the problem has been done. Then nonlinear function is linearized by the help of quasilinearization. Further we discuss how to implement this technique involving Haar wavelet to the so called generalized Burgers' equation. In section \ref{sec7}, $L^2$ error estimate is analysed. In section \ref{sec8}, numerical solution of the generalized Burgers' equation by proposed method are tabulated and illustrated graphically for test problems. Lastly in section \ref{sec9}, we conclude the paper.

\section{Preliminary}
\label{sec2}
For the numerical solutions of differential equations, integro differential equations and integral equations, one of the simplest mathematical tool that is being used, is Haar wavelet method. Haar wavelets are one of the simplest wavelet among the various type of wavelets. It have been used from $1910$ when these were established by the Hungarian mathematicin Alfred Haar. Haar functions are unit step functions which takes three values $0,1$ and $-1$. Haar function is the oldest and simplest orthonormal wavelet with compact support. For $x\in [0,1)$, the family of Haar wavelet is defined as:

\begin{equation}
\label{hq1}
h_{i}(x)=
\begin{cases}
1,&x\in[\eta_{1},\eta_{2})\\
-1,& x\in[\eta_{2},\eta_{3})\\
0,& \text{otherwise}
\end{cases}
\end{equation}

where 
\begin{equation}
\label{in1}
\eta_{1}=\frac{k}{m},\hspace{0.5in} \eta_{2}=\frac{k+0.5}{m} \hspace{0.5in}\text{and}\hspace{0.5in} \eta_{3}=\frac{k+1}{m}.
\end{equation}
In equation \eqref{in1}, we have the integer $m=2^{j},~j=0,1,2,...,J$ defines the level of wavelets and $k$ is the translation parameter given as $k=0,1,2,...,m-1$. The indices in $h_{i}$ in equation \eqref{hq1} can be evaluated from the formula $i=m+k+1$. For the minimal values of $k=0,m=1$, we have the minimal value of $i=2$. For the maximul value of $k=m-1$ we have the maximum value of  $i$ given by $i=2M=2^{J+1}$, $J$ is the maximum resolutions.
$h_{1}(x)$ is the scaling function which is the member of the Haar function defined as
\begin{equation}
\label{hq2}
h_{1}(x)=
\begin{cases}
1,&x\in[0,1)\\
0,& \text{otherwise}.
\end{cases}
\end{equation}
For the solution of any differential equation, we need to integrate the Haar function which is given in the following integrals:
\begin{eqnarray}
\nonumber p_{\sigma,l}(x)&=&\underset{\sigma-times}{\underbrace{\int_{0}^{x} \int_{0}^{\xi_{\sigma -1}}...\int_{0}^{\xi_{1}}h_{i}(\xi)}}~d\xi~ d\xi_{1}...d\xi_{\sigma -1}\\
\label{hq3}
&=&\frac{1}{(\sigma -1)!}\int_{0}^{x}(x-\xi)^{\sigma-1}h_{i}(\xi)~d\xi,
\end{eqnarray}
where $\sigma=1,2,...,n;~~~l=1,2,...,2M.$ By considering the equation \eqref{hq1} the integrals \eqref{hq3} can be determine analytically. Thus we have
\begin{equation}
\label{hq4}
p_{\sigma,i}=
\begin{cases}
0,&x<\eta_{1}\\
\frac{(x-\eta_{1})^{\sigma}}{\sigma !},&x\in[\eta_{1},\eta_{2})\\
\frac{(x-\eta_{1})^{\sigma}}{\sigma !}-\frac{(x-\eta_{2})^{\sigma}}{\sigma !},&x\in[\eta_{2},\eta_{3})\\
\frac{(x-\eta_{1})^{\sigma}}{\sigma !}-\frac{(x-\eta_{2})^{\sigma}}{\sigma !}+\frac{(x-\eta_{3})^{\sigma}}{\sigma !},&x>\eta_{3}.
\end{cases}
\end{equation}
The above formula is applicable for $i>1$. For the case $\sigma=1$ and $\sigma=2$, we have the following
\begin{equation}
\label{i1}
p_{1,i}(x)=
\begin{cases}
x-\eta_{1},&x\in[\eta_{1},\eta_{2})\\
\eta_{3}-x,& x\in[\eta_{2},\eta_{3})\\
0,& \text{otherwise}
\end{cases}
\end{equation}
\begin{equation}
\label{i2}
p_{2,i}(x)=
\begin{cases}
\frac{(x-\eta_{1})^2}{2},&x\in[\eta_{1},\eta_{2})\\
\frac{1}{4m^2}-\frac{(\eta_{3}-x)^2}{2},& x\in[\eta_{2},\eta_{3})\\
\frac{1}{4m^2},& x\in[\eta_{3},1]\\
0,& \text{otherwise}.
\end{cases}
\end{equation}

Since all the Haar wavelets are orthogonal to each other i.e.
\begin{equation}
\int_{0}^{1}h_{i}(x)h_{l}(x)dx =2^{-j}\delta_{i,l}=
\begin{cases}
2^{-j},&i=l=2^j+k\\
0,&i\neq l,
\end{cases}
\end{equation}
therefore any function $w(x)\in L^2[0,1)$ can be expressed as the sum of infinite series form
\begin{equation}
w(x)=\sum_{i=1}^{\infty}c_{i}h_{i}(x)
\end{equation}
where the coefficient $c_{i}$ can be calculated by
\begin{equation}
c_{i}=2^j\int_{0}^{1}w(x)h_{i}(x),
\end{equation}
where $i=2^j+k,~j\geq 0,~ 0\leq k< 2^j. $
The above series contains infinite number of terms. In case, the  function $w(x)$ is piecewise constant by itself or approximated as piecewise constant  then  the series terminates at finite terms and can be expressed as
\begin{equation}
w(x)=\sum_{i=1}^{2M}c_{i}h_{i}(x)=C^{T}_{(2M)}h_{(2M)}(x),
\end{equation}
Here $T$ represent transpose and $2M=2^{J+1}$ is the lenght of the vector given by
\begin{eqnarray}
\nonumber C^T_{(2M)}&=&[c_{1},c_{2},...,c_{(2M)}]\\
h_{(2M)}(x)&=&[h_{(1)}(x),h_{(2)}(x),...,h_{(2M)}(x)]^T.
\end{eqnarray}
\section {Derivation of the scheme}
\label{sec4}
Since the Haar wavelets described only for $x\in[0,1]$ so first we transfer the interval $x_{*}\in[a,b]$ in to unit interval  $x\in[0,1]$. Let us use the expression 
$x=\frac{x_{*}-a}{L},\text{where}~L=b-a$. Hence equation \eqref{eq1}  becomes
\begin{equation}
\label{eq2}
\frac{\partial w}{\partial t}+\frac{1}{L}w^{\mu}w_{x}=\frac{\nu}{L^2} w^{\delta}w_{xx},\hspace{0.1in}0\leq x\leq 1,~~t\geq 0,
\end{equation}
with the boundary conditions (BCs)
\begin{equation}
w(0,t)=f_{1}(t) \hspace{0.5in} w(1,t)=f_{2}(t),\hspace{0.5 in} t\in [0,T]
\end{equation}
and the initial conditions 
\begin{equation}
w(x,0)=f(x), \hspace{0.5in} x\in[0,1].
\end{equation}
Now, we apply finite forward difference for the time and average in time for $w^{\mu} w_{x}$ and $w^{\delta}w_{xx}$ in equation \eqref{eq2}, we have 
\begin{eqnarray}
\label{eq3}
	\frac{w_{j+1}-w_{j}}{\Delta t}+\frac{1}{2L}\Big[w^{\mu}_{j+1}(w_{j+1})_{x}+w^{\mu}_{j}(w_{j})_{x}\Big]=\frac{\nu}{2L^2}\Big[ w^{\delta}_{j+1}(w_{j+1})_{xx}+w^{\delta}_{j}(w_{j})_{xx}]~~~~~0\leq j\leq N-1,
	\end{eqnarray}
	with the BCs
	\begin{eqnarray}
w_{j+1}(0)=f_{1}(t_{j+1})\hspace{0.5in} w_{j+1}(1)=f_{2}(t_{j+1}) ~~~~~0\leq j\leq N-1,
\end{eqnarray}
and initial condition
\begin{equation}
w_{0}=f(x),
\end{equation}
where $t_{j+1}=\Delta t (j+1),~~N\Delta t=T$($\Delta t$ is  time step size and $T$ is final time) and $w_{j+1}$ is the  approximated solution  at $(j+1)$th time level. Equation \eqref{eq3} can be expressed as follows
\begin{eqnarray}
\label{eq4}
 \frac{\nu \Delta t }{2L^2} w^{\delta}_{j+1}(w_{j+1})_{xx}-\frac{\Delta t}{2L}w^{\mu}_{j+1}(w_{j+1})_{x}-w_{j+1}=-w_{j}+\frac{\Delta t}{2L}w^{\mu}_{j}(w_{j})_{x} -\frac{\nu\Delta t}{2L^2}w^{\delta}_{j}(w_{j})_{xx}~~ 0\leq j\leq N-1,
\end{eqnarray}
	with the BCs
\begin{eqnarray}
w_{j+1}(0)=f_{1}(t_{j+1})\hspace{0.5in} w_{j+1}(1)=f_{2}(t_{j+1})~~~~~0\leq j\leq N-1,
\end{eqnarray}
and initial condition
\begin{equation}
w_{0}=f(x).
\end{equation}
Equation \eqref{eq4} is the non-linear ordinary differential equations at $(j+1)$th time level.

Nonlinearity in equation  \eqref{eq4} can be handled by several methods. One of the possible method is  quasilinearization process \cite{bellman1965nonlinear}. For $\mu\geq 1,\delta \geq 1$ the nonlinear term in equation \eqref{eq4} is replaced by the following equations
\begin{eqnarray}
w^{\mu}_{j+1}(w_{j+1})_{x}=\mu w_{j}^{\mu-1}w_{j+1}(w_{j})_{x}+w_{j}^{\mu}(w_{j+1})_{x}-\mu w^{\mu}_{j}(w_{j})_{x}\\
w^{\delta}_{j+1}(w_{j+1})_{xx}=\delta w_{j}^{\delta-1}w_{j+1}(w_{j})_{xx}+w_{j}^{\delta}(w_{j+1})_{xx}-\delta w^{\delta}_{j}(w_{j})_{xx}
\end{eqnarray}
Thus, we have the following equations 
\begin{multline}
\label{eq5}
\frac{\nu \Delta t }{2L^2} w^{\delta}_{j}(w_{j+1})_{xx}-\frac{\Delta t}{2L}w^{\mu}_{j}(w_{j+1})_{x}\\
+\frac{\nu \Delta t }{2L^2}\delta w^{(\delta-1)}_{j}(w_{j})_{xx}w_{j+1}
-\frac{\Delta t}{2L}\mu w^{(\delta-1)}_{j}(w_{j})_{x}w_{j+1}-w_{j+1}\\
=-w_{j}+\frac{\Delta t}{2L}(1-\mu)w^{\mu}_{j}(w_{j})_{x}-\frac{\nu \Delta t }{2L^2}(1-\delta) w^{\delta}_{j}(w_{j})_{xx},
\end{multline} 
 with the boundary conditions
\begin{eqnarray}
w_{j+1}(0)=f_{1}(t_{j+1}),\hspace{0.5in} w_{j+1}(1)=f_{2}(t_{j+1}), ~~~~~0\leq j\leq N-1,
\end{eqnarray}
and initial condition
\begin{equation}
w_{0}=f(x).
\end{equation}

Now, we descretize the second order spatial derivative present in equation \eqref{eq5}  using the Haar wavelets as follows 
\begin{eqnarray}
\label{eq6}
(w_{j+1})_{xx}(x)=\sum_{i=1}^{2M}c_{i}h_{i}(x)=C^{T}_{2M}h_{2M}(x)
\end{eqnarray}
Integrating the equation \eqref{eq6} from $0$ to $x$, we have
\begin{equation}
\label{eq7}
(w_{j+1})_{x}(x)=\sum_{i=1}^{2M}c_{i}p_{i,1}(x)+(w_{j+1})_{x}(0),
\end{equation}
$w_{j+1}(0)$ is unknown in the equation \eqref{eq7}. To find the this we integrate \eqref{eq7} from $0$ to $1$ and using BCs, we have
\begin{equation}
\label{eq8}
(w_{j+1})_{x}(x)=\sum_{i=1}^{2M}c_{i}\Big[p_{i,1}(x)-p_{i,2}(1)\Big]+f_{2}(t_{j+1})-f_{1}(t_{j+1}),
\end{equation}
again integrating equation \eqref{eq8} from $0$ to $x$, we get
\begin{eqnarray}
\label{eq9}
w_{j+1}(x)=\sum_{i=1}^{2M}c_{i}\Big[p_{i,2}(x)-p_{i,2}(1)\Big]+x(f_{2}(t_{j+1})-f_{1}(t_{j+1}))+f_{1}(t_{j+1}).
\end{eqnarray}
Now, putting equations \eqref{eq6}-\eqref{eq9} in  \eqref{eq5}, we have
\begin{multline}
 \label{eq10} \sum_{i=1}^{2M}c_{i}\Big[ \frac{\nu \Delta t }{2L^2} w^{\delta}_{j}h_{i}(x)-\frac{\Delta t}{2L}w^{\mu}_{j}\Big(p_{i,1}(x)-p_{i,2}(1)\Big)\\+\Big(\frac{\nu \Delta t }{2L^2} \delta w^{(\delta-1)}_{j}(w_{j})_{xx}-\frac{\Delta t}{2L}\mu w^{(\mu-1)}_{j}(w_{j})_{x}-1\Big)  \Big(p_{i,2}(x)-p_{i,2}(1)\Big)\Big]\\=-w_{j}+\frac{\Delta t}{2L}(1-\mu)w^{\mu}_{j}(w_{j})_{x}-\frac{\nu \Delta t }{2L^2}(1-\delta) w^{\delta}_{j}(w_{j})_{xx}+\frac{\Delta t}{2L}w^\mu_{j} \Big(f_{2}(t_{j+1})-f_{1}(t_{j+1})\Big)\\-\Big(\frac{\nu \Delta t }{2L^2} \delta w^{(\delta-1)}_{j}(w_{j})_{xx}-\frac{\Delta t}{2L}\mu w^{(\mu-1)}_{j}(w_{j})_{x}-1\Big) \Big(x(f_{2}(t_{j+1})-f_{1}(t_{j+1}))+f_{1}(t_{j+1})\Big),
\end{multline}
where $p_{2,i}(1)$ can easily  calculated from equation \eqref{i2} and are given by
\begin{equation}
p_{2,i}(1)=
\begin{cases}
0.5,&i= 1\\
\frac{1}{4m^2},& i>1.
\end{cases}
\end{equation}
Now, let us take the collocation points $x_{k}=\frac{k-0.5}{2M},k=1,2,...,2M$ and applying discretization on equation  \eqref{eq10}, we get the following linear system 
\begin{multline}\label{eq11} 
\sum_{i=1}^{2M}c_{i}\Big[ \frac{\nu \Delta t }{2L^2} w^{\delta}_{j}h_{i}(x_{k})-\frac{\Delta t}{2L}w^{\mu}_{j}\Big(p_{i,1}(x_{k})-p_{i,2}(1)\Big)\\+\Big(\frac{\nu \Delta t }{2L^2} \delta w^{(\delta-1)}_{j}(w_{j})_{xx}-\frac{\Delta t}{2L}\mu w^{(\mu-1)}_{j}(w_{j})_{x}-1\Big)  \Big(p_{i,2}(x_{k})-p_{i,2}(1)\Big)\Big]\\
=-w_{j}+\frac{\Delta t}{2L}(1-\mu)w^{\mu}_{j}(w_{j})_{x}-\frac{\nu \Delta t }{2L^2}(1-\delta) w^{\delta}_{j}(w_{j})_{xx}+\frac{\Delta t}{2L}w^{\mu}_{j} \Big(f_{2}(t_{j+1})-f_{1}(t_{j+1})\Big)\\-\Big(\frac{\nu \Delta t }{2L^2} \delta w^{(\delta-1)}_{j}(w_{j})_{xx}-\frac{\Delta t}{2L}\mu w^{(\mu-1)}_{j}(w_{j})_{x}-1\Big) \Big(x_{k}(f_{2}(t_{j+1})-f_{1}(t_{j+1}))+f_{1}(t_{j+1})\Big).
\end{multline}
 By solving the above linear system, we can obtain the wavelets coefficient $C^T_{2M}$. For the first time step the value of $(w_{j})$, $(w_{j})_{x}$,$(w_{j})_{xx}$ can be taken from initial conditions. For the next time step the value of $(w_{j})$, $(w_{j})_{x}$,$(w_{j})_{xx}$ are calculated by solving the above equation for the wavelet coefficient $C^T_{2M}$ and putting in equation \eqref{eq9},\eqref{eq8} and \eqref{eq6}. The value of $(w_{j})$, $(w_{j})_{x}$,$(w_{j})_{xx}$  for each time step can be obtained in the same way. To start the iterations, we use $w_{0}(x_{k})=f(x_{k}), (w_{x})_{0}(x_{k})=f'(x_{k}), (w_{xx})_{0}(x_{k})=f''(x_{k})$.
\section{$L^2$ Error}
\label{sec7}
For the convergence of the projected method, we analyze the asymptotic expression of the equation \eqref{eq9} and the corresponding equation is below
\begin{eqnarray}
w(x)=\sum_{i=1}^{\infty}c_{i}\Big[p_{i,2}(x)-p_{i,2}(1)\Big]+x(f_{2}(t_{j+1})-f_{1}(t_{j+1}))+f_{1}(t_{j+1}).
\end{eqnarray}

\begin{lemma}
	\label{lem1}
	Let  us assume that $w(x)\in L^2(R)$ with $| w_{x}(x)| \leq K,~ \forall~ x\in(0,1); ~K > 0 $ and $w(x)=\sum_{i=0}^{\infty}c_{i}h_{i}(x)$. Then $|c_{i}|\leq K2^{-(3j-2)/2}.$
\end{lemma}
\begin{proof}
See \cite{ray2012haar}.
\end{proof}
\begin{lemma}
	\label{lem2}
	Let $w(x)\in L^2(R)$ be a continuous function in the interval $(0,1)$. Then at $ J$ th level, the error norm is bounded by
	\begin{eqnarray}
	||E_{J}||^2_{2}\leq \frac{K^2}{12}2^{-2J},
	\end{eqnarray} 
	where $| w_{x}(x)| \leq K, ~\forall~ x\in(0,1);~ K > 0 ,~M$ is the positive  given by $M=2^J.$ 
\end{lemma}
\begin{proof}
	See \cite{ray2012haar}.
\end{proof}
\begin{theorem}
	Let $w(x)$ and $w_{2M}$ are the exact and approximated solution of the equation \eqref{eq9},then
	\begin{eqnarray}
	||E_{J}||_{2}=||w(x)-w_{2M}(x)||_{2}\leq2K\Big(\frac{2^{-(\frac{5}{2}(J+1)+1)}}{1-2^{-5/2}}\Big).
	\end{eqnarray}
\begin{proof}	
	\begin{eqnarray}
	\nonumber&&||E_{J}||^2_{2}=\int_{0}^{1}\Big[\sum_{j=J+1}^{\infty}\sum_{l=J+1}^{\infty}c_{2^j+k+1}\Big(p_{2,2^j+k+1}(x)-xp_{2,2^j+k+1}(1)\Big)\Big]^2dx\\
	&& =\sum_{j=J+1}^{\infty}\sum_{k=0}^{2^j-1}\sum_{l=J+1}^{\infty}\sum_{s=0}^{2^l-1}c_{2^j+k+1}c_{2^l+s+1}\int_{0}^{1}\Big[\Big(p_{2,2^j+k+1}(x)-xp_{2,2^j+k+1}(1)\Big) \\\label{er1}&&\hspace{3in}\Big(p_{2,2^l+s+1}(x)-xp_{2,2^l+s+1}(1)\Big)\Big]dx
	\end{eqnarray}
	First we will evaluate the  upper bound for the function $p_{2,i}(x)$ in all the  subinterval $[0,1]$. We have $p_{2,i}=0$ for $x\in[0,\eta_{1}]$. In the interval $[\eta_{1},\eta_{2}]$, function $p_{2,i}(x)$ is monotonic increasing and its maximum value can be obtained by putting $x=\eta_{2}$ and therefor
	 \begin{eqnarray}
	p_{2,i}(x)=p_{2,2^j+k+1}(x)\leq \frac{(\eta_{2}-\eta_{1})^2}{2}=\frac{1}{2}\Big(\frac{1}{2^{(j+1)}}\Big)^2,~~\forall x\in[\eta_{1},\eta_{2}].
	\end{eqnarray}
In the interval $[\eta_{2},\eta_{3}]$ it can be easily prove that function $p_{2,i}$ is monotonically increasing by using the equation \eqref{in1},\eqref{hq4} and the condition $\frac{\partial p_{2,i}(x)}{\partial x}>0$ if $x<\eta_{3}$ which is true. Hence $p_{2,i}(x)$  attains its maximum value at the end point $x=\eta_{3}$. Hence
\begin{eqnarray}
	p_{2,i}(x)=p_{2,2^j+k+1}(x)\leq \Big(\frac{1}{2^{(j+1)}}\Big)^2,~~\forall x\in[\eta_{2},\eta_{3}].
\end{eqnarray}
For the subinterval $[\eta_{3},1]$, $p_{2,i}(x)$ can be express as (Equation (22) \cite{majak2015convergence}) 
\begin{eqnarray}
	p_{2,i}(x)= p_{2,2^j+k+1}(x)=\Big(\frac{1}{2^{(j+1)}}\Big)^2. 
\end{eqnarray}
Thus the function $p_{2,i}(x)$ has an upper bound in $[0,1]$ given by
\begin{eqnarray}
p_{2,i}(x)=p_{2,2^j+k+1}(x)\leq\Big(\frac{1}{2^{(j+1)}}\Big)^2~\forall x\in[0,1].
\end{eqnarray}
Now, we have 
\begin{eqnarray}
\nonumber\Big(p_{2,i}(x)-xp_{2,i}(1)\Big)&\leq& |p_{2,i}(x)|+|x||p_{2,i}(1)|\\
\nonumber&\leq&|p_{2,i}(x)|+|p_{2,i}(1)|\\ \label{er2}
&\leq&2\Big(\frac{1}{2^{(j+1)}}\Big)^2.
\end{eqnarray}
Similarly, we have 
\begin{eqnarray}
\label{er3}
\Big(p_{2,l}(x)-xp_{2,l}(1)\Big)&\leq&2\Big(\frac{1}{2^{(l+1)}}\Big)^2.	
\end{eqnarray}
	Putting \eqref{er2} and \eqref{er3} in equation \eqref{er1}, we get
	\begin{eqnarray}
\nonumber	||E_{J}||^2_{2}&\leq& 4\sum_{j=J+1}^{\infty}\sum_{k=0}^{2^j-1}\sum_{l=J+1}^{\infty}\sum_{s=0}^{2^l-1}c_{2^j+k+1}c_{2^l+s+1}\Big(\frac{1}{2^{(j+1)}}\Big)^2\Big(\frac{1}{2^{(l+1)}}\Big)^2\\
\nonumber &\leq& 4K^{2}\sum_{j=J+1}^{\infty}\sum_{l=J+1}^{\infty}2^{-(\frac{3}{2}j-1)}2^j\Big(\frac{1}{2^{(j+1)}}\Big)^22^{-(\frac{3}{2}l-1)}2^l\Big(\frac{1}{2^{(l+1)}}\Big)^2 ~~~\text{by lemma \ref{lem1}} \\
\nonumber &\leq& 4K^2\sum_{j=J+1}^{\infty}2^{-(\frac{3}{2}l-1)}2^j\Big(\frac{1}{2^{(j+1)}}\Big)^2\frac{2^{-(\frac{5}{2}(J+1)+1)}}{1-2^{-5/2}}\\
\nonumber &\leq& 4K^2\Big(\frac{2^{-(\frac{5}{2}(J+1)+1)}}{1-2^{-5/2}}\Big)^2.
	\end{eqnarray}
	Hence
	\begin{eqnarray}
	\label{er4}
	||E_{J}||_{2}\leq2K\Big(\frac{2^{-(\frac{5}{2}(J+1)+1)}}{1-2^{-5/2}}\Big).
	\end{eqnarray}
\end{proof}
\end{theorem}
	It is clear that equation \eqref{er3} indecates that the error bound is propotional to the level of resolutions $J$ of the Haar wavelet. Also $||E_{J}||_{2}\rightarrow 0$ as $J\rightarrow \infty$. Thus the proposed scheme converges to the solutions as $J$ approaches to infinity.
\section{Numerical Illustration}
\label{sec8}
In this part, we measure the efficiency of the proposed method  by taking some example with the help of mean square root error norm $L_{2}$  and maximum error norm $L_{\infty}$ defined as
 \begin{eqnarray*}
 	L_{2}=\sqrt{\Delta x\sum_{j=0}^{2M}\mid w^{exact}_{j}-(w_{2M})_{j} \mid^{2}},~~~L_{\infty}=\Arrowvert w^{exact}-w_{2M} \Arrowvert _{\infty}= \mathop{max}_{j}\mid w^{exact}_{j}-(w_{2M})_{j} \mid,
 \end{eqnarray*}
where $w_{2M}$ is the approximated result by Haar Wavelet method.
\subsection{Test Problem 1}
\label{prob1}
Let us take $\mu=2,\delta=0$, in equation \eqref{eq1} and the BCs
\begin{eqnarray}
w(0,t)=0=w(1,t)~~t>1,
\end{eqnarray} 
 with the initial condition
\begin{eqnarray}
w(1,t)=\frac{x_{*}}{1+e^{\frac{1}{c_{0}}(\frac{x_{*}^2}{4\nu} )}}, ~x_{*}\in(0,1)
\end{eqnarray}
which is obtained from the exact solution \cite{bratsos2010fourth}
\begin{eqnarray}
w(x_{*},t)=\frac{\frac{x_{*}}{t}}{1+\sqrt{t}/c_{0}~e^{(\frac{x_{*}^2}{4\nu t})}}~t\geq 1, x_{*}\in(0,1)
\end{eqnarray}
where $0< c_{0}<1.$
\par For the comparision purpose, in table \ref{TN1},  we take $\Delta t=0.01,2M=16,\nu=0.01$ and $C_{0}=0.5$ and we compute the $L_{\infty}$-error and $L_{2}$-error at $T=2$ and $T=4$. We observe that the result  by present scheme is better than the existing result published in \cite{bratsos2018exponential}-\cite{saka2008numerical} even though $\Delta t$ and $\Delta x$ taken in present scheme is very large compare to the value of $\Delta t$ and $\Delta x$ taken in \cite{bratsos2018exponential}-\cite{saka2008numerical} for $\nu=0.01$. In table \ref{TN1.1}, the result are computed for small value of $\nu=0.001$ for $\Delta t=0.01$ at $T=2$ and $T=4$. Here we take $2M=32$ and $C_{0}=0.5$ and again compare the result obtained by the present method with the help of $L_{2}$ and $L_{\infty}$-error and observed that the error produced by the present method are less than the error produced by the method in the published work \cite{bratsos2018exponential}-\cite{saka2008numerical} for the value of $\Delta t$ and $\Delta x$ taken in present scheme which is very large compare to the value of $\Delta t$ and $\Delta x$ taken in \cite{bratsos2018exponential}-\cite{saka2008numerical}. \par In fig. \ref{F1}, we plot the numerical result (Left) for $\nu=0.005,2M=16$ and $C_{0}=0.5$ and find that the computed result follows the physical behaviour of the solution at different time for the time step $\Delta t=0.01$.  In fig \ref{F1} (Right) the absolute error is plotted and observe that the absolute error at different discrete point are less than $0.0005$ which are accetable. 
\begin{table}[!ht]
	\centering
	\caption{Comparision of numerical result with the existing result by the help of $L_{\infty}$ and $L_{2}$ error of the problem \eqref{prob1} at $T=2,4$ for $\nu=0.01,\Delta t=0.01$ and $C_{0}=0.5,J=3.$}
	\label{TN1}
	\begin{tabular}{lllllll}
		\hline
		&                &            & \multicolumn{2}{l}{T=2}   \hspace{1.5in}         & \multicolumn{2}{l}{T=4}            \\ 
	& &		&	\multicolumn{2}{l}{\noindent\rule{4.5cm}{0.6pt}}        & \multicolumn{2}{l}{\noindent\rule{4.5cm}{0.6pt}} \\
		$\nu=0.01$    & $\Delta x$ & $\Delta t$ & $L_{\infty}-error$ & $L_{2}-error$ & $L_{\infty}-error$ & $L_{2}-error$ \\ \hline
		Present       & 1/16           & 0.01       & 0.76E-03           & 0.347E-03     & 0.582E-03          & 0.311E-03     \\
		\cite{bratsos2018exponential} & 1/100          & 0.00001    & 0.81387E-03        & 0.38291E-03   & 0.60474E-03        & 0.31718E-03   \\
		\cite{ramadan2005numerical}   & 1/100          & 0.00001    & 1.21698E-03        & 0.52308E-03   & 0.93136E-03        & 0.51625E-03   \\
		\cite{ramadan12005numerical}   & 1/100          & 0.00001    & 1.70309E-03        & 0.79043E-03   & 0.99645E-03        & 0.55767E-03   \\
		\cite{saka2008numerical}       & 1/100          & 0.00001    & 0.81680E-03        & 0.37932E-03   & 0.60537E-03        & 0.31724E-03  \\ \hline
	\end{tabular}
\end{table}
\begin{table}[!ht]
	\centering
	\caption{Comparision of numerical result with the existing result by the help of $L_{\infty}$ and $L_{2}$ error of the problem \eqref{prob1} at $T=2,4$ for $\nu=0.01,\Delta t=0.01$ and $C_{0}=0.5,J=3.$}
	\label{TN1.1}
	\begin{tabular}{lllllll}
		\hline
	&                &            & \multicolumn{2}{l}{T=2}   \hspace{1.5in}         & \multicolumn{2}{l}{T=4}            \\ 
	& &		&	\multicolumn{2}{l}{\noindent\rule{4.5cm}{0.6pt}}        & \multicolumn{2}{l}{\noindent\rule{4.5cm}{0.6pt}} \\
	$\nu=0.001$    & $\Delta x$ & $\Delta t$ & $L_{\infty}-error$ & $L_{2}-error$ & $L_{\infty}-error$ & $L_{2}-error$ \\ \hline
		Present       & 1/32           & 0.01       & 0.2236E-03         & 0.054998E-03  & 0.1823E-03         & 0.575806E-03  \\
		\cite{bratsos2018exponential} & 1/100          & 0.00001    & 0.26595E-03        & 0.07173E-03   & 0.19549E-03        & 0.05727E-03   \\
		\cite{ramadan2005numerical}   & 1/100          & 0.00001    & 0.27967E-03        & 0.06703E-03   & 0.21856E-03        & 0.06670E-03   \\
		\cite{ramadan12005numerical} & 1/100          & 0.00001    & 0.81852E-03        & 0.18355E-03   & 0.35635E-03        & 0.11441E-03   \\
		\cite{saka2008numerical}      & 1/100          & 0.00001    & 0.26094E-03        & 0.06811E-03   & 0.19288E-03        & 0.05652E-03  \\ \hline
	\end{tabular}
\end{table}
\begin{figure}[!ht]
	\includegraphics[height=2.5in,width=3.5in]{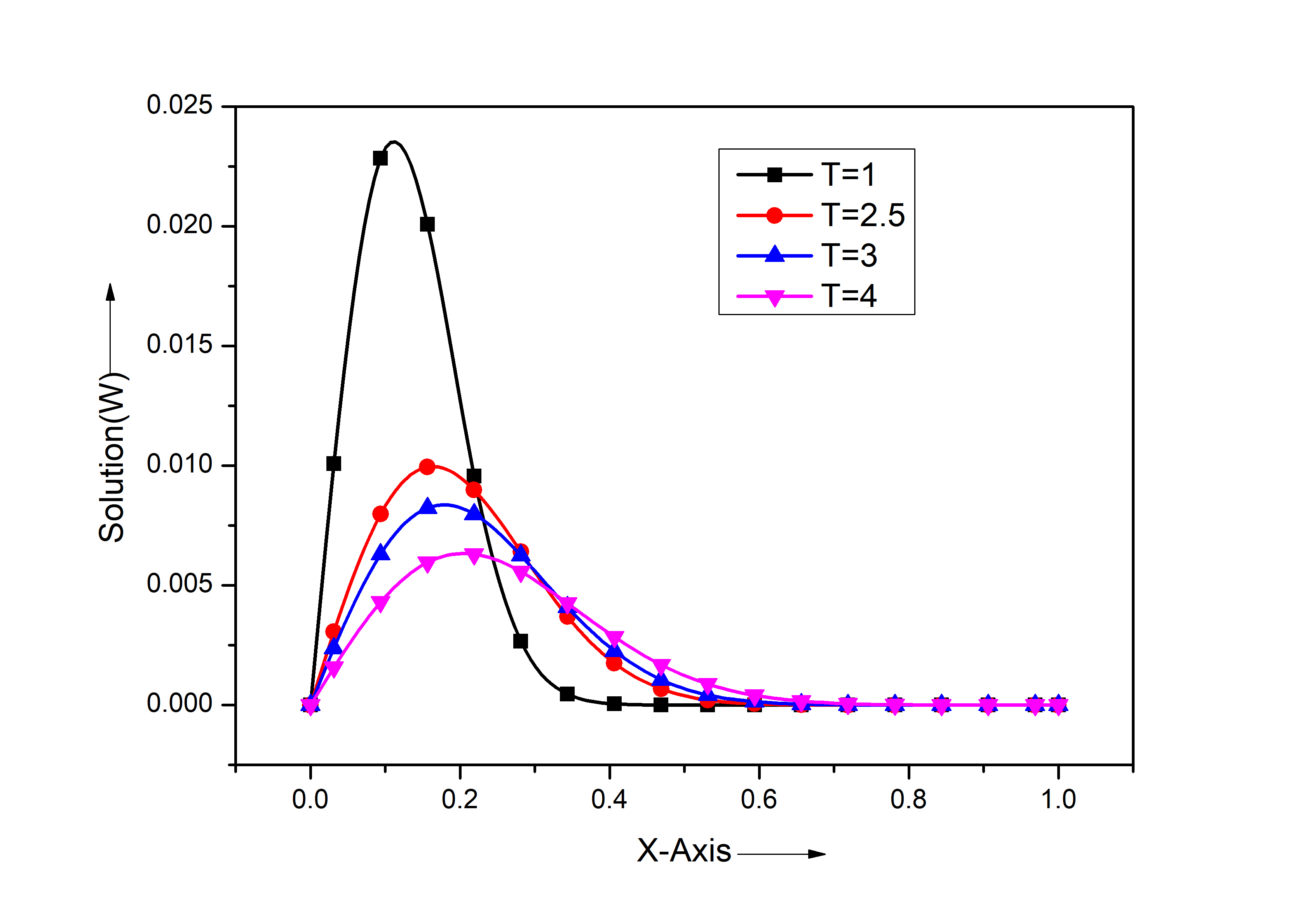}
	\includegraphics[height=2.5in,width=3.5in]{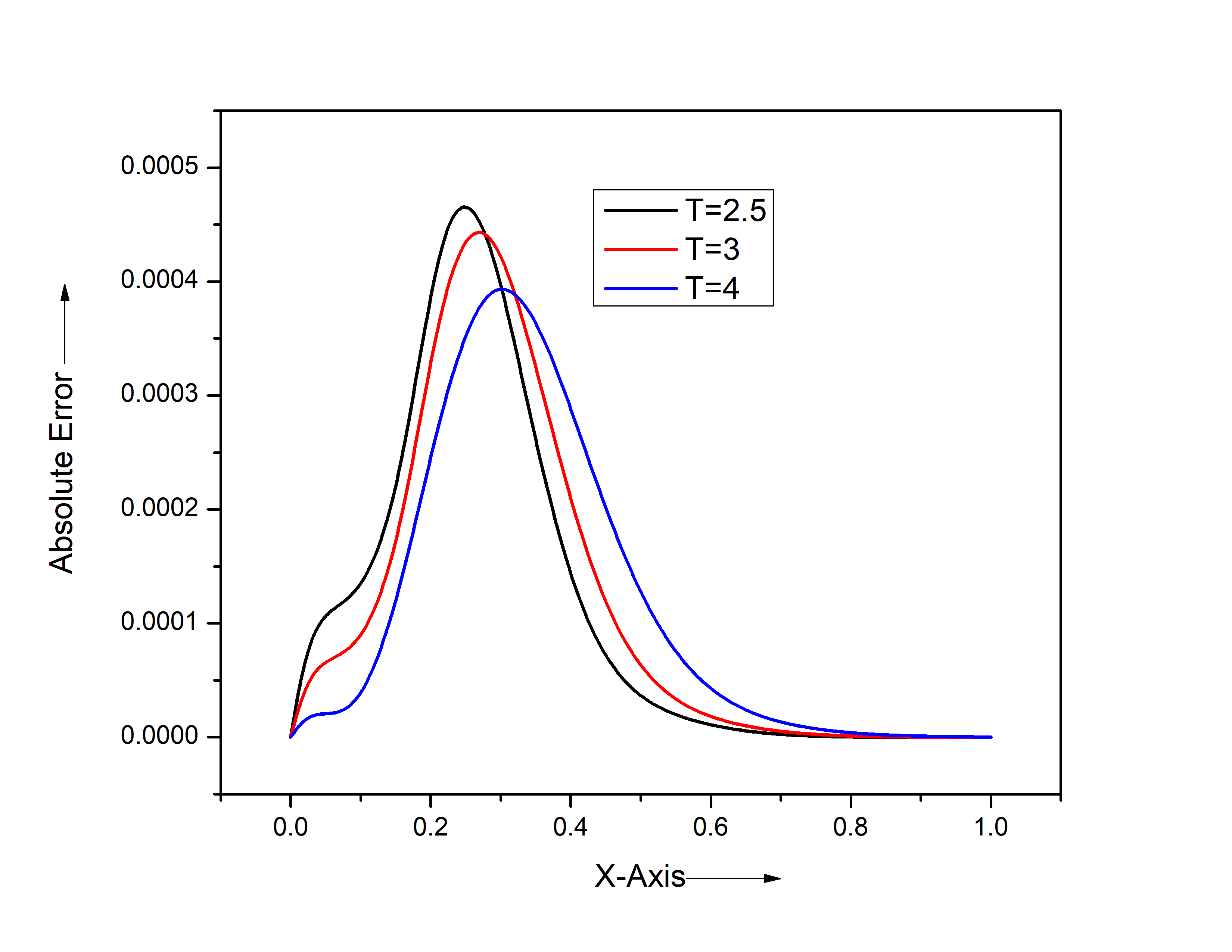}
	\caption{Numerical solution (left) and absolute error(right) for the problem \eqref{prob1} for the time step $\Delta t=0.01,~\nu=0.005$, $J=3$ at different time $T$.}
	\label{F1}
\end{figure}
\subsection{Test Problem 2}
\label{prob2}
 Let us take $\mu=1,\delta=1$  with the BCs
 \begin{eqnarray}
 w(0,t)=0=w(1,t)~~t>0
 \end{eqnarray}
 and  initial condition
 \begin{eqnarray}
 w(x_{*},0)=\frac{1-e^{x_{*}/\nu}+(e^{1/\nu}-1)x_{*}}{\sigma},~x_{*}\in(0,1),
 \end{eqnarray}
 which is obtained from the exact solution \cite{mickens1994nonstandard}
 \begin{eqnarray}
  w(x_{*},t)=\frac{1-e^{x_{*}/\nu}+(e^{1/\nu}-1)x_{*}}{(e^{1/\nu}-1)t+\sigma},t\geq 0,
 \end{eqnarray}
 where $\sigma >0$ is a parameter. 
 We compute the $L_{\infty} $ and $L_{2}$-norm error. In table \ref{TN2}, we summerize the result computed by the present method for $\nu=1,\sigma =2$ with the time step $\Delta t=0.001,0.01$ and $2M =8$ and $32$ at different time $T$. We can see in  table \ref{TN2} that as the value of $2M$ is increased, $L_{\infty}$ and $L_{2}$ error  decrease and hence we can say that  when $J\rightarrow \infty$  the error will tend to zero.\par In fig.\ref{F2} (Left), the exact solution and the numerical solution are ploted for the time step $\Delta t=0.001,\nu=1$, $ \sigma=1$ at different time $T$ and we see that the numrical result are very closed to the exact solution. In fig. \ref{F2} (Right) absolute error is ploted at  discrete point for the different time with the time step  $\Delta t=0.001,\nu=1$ and $ \sigma=1$. We see that the absolute error are very small and less than $0.0000012$ and is maximum near the point $0.7$  for all different time. In fig \ref{F2.2} (Left) we have graphed the numerical solution and exact solution of the problem \eqref{prob2} for small value of $\nu=0.01, \Delta t=0.01$ at different $T$ for $2M=32$. We can see that the numerical solution is almost same as the exact solution exept near the boundary point $1$. In fig \ref{F2.2} (Right), we have graphed the absolute error of the problem at different time $T$ for $\nu=0.01, \Delta t=0.01$ and $2M=32$. It is observed that the absolute error throughout the domain is almost zero except near the boundary point $1$. This error can be reduced by taking the more number of greed points as seen in the table $\ref{TN2}.$
\begin{table}[!ht]
	\centering
	\caption{$L_{\infty}$ and $L_{2}$ error of the problem \eqref{prob2} at different $T$ for $J=2,4,~\nu=1$ and $\sigma=2.$}
	\label{TN2}
	\begin{tabular}{lllll}
		\hline
		&       \hspace{0.75in}             & $T=0.01$   \hspace{0.75in}      & $T=0.1$     \hspace{0.75in}    & $T=0.2$         \\
		&                    & $\Delta t=0.001$ & $\Delta t=0.01$ & $\Delta t=0.01$ \\ \hline
		\multirow{2}{*}{$J=2$} & $L_{\infty}~error$ & 1.1533E-06
		      & 9.9506E-06      & 1.73036E-05     \\ 
		& $L_{2}~error$      & 8.18486E-07      & 7.00077E-06     & 1.22587E-05     \\ \hline
		\multirow{2}{*}{$J=4$} & $L_{\infty}~error$ & 7.31654E-08      & 6.26645E-07     & 1.09634E-06     \\ 
		& $L_{2}~error$      & 5.12615E-08      & 4.40074E-07     & 7.72171E-07  \\  \hline
	\end{tabular}
\end{table}
\begin{figure}[!ht]
	\includegraphics[height=2.5in,width=3.5in]{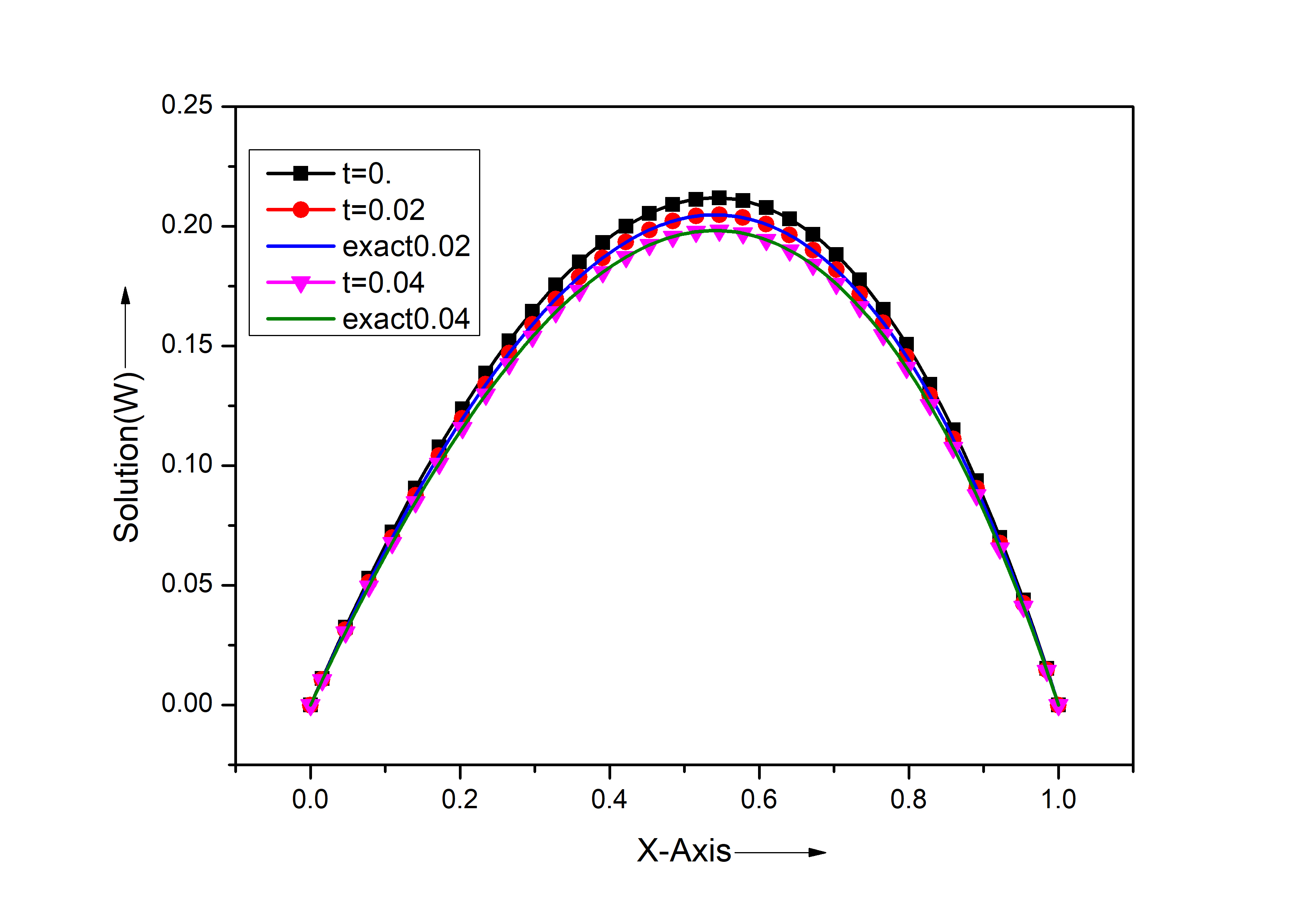}
	\includegraphics[height=2.5in,width=3.5in]{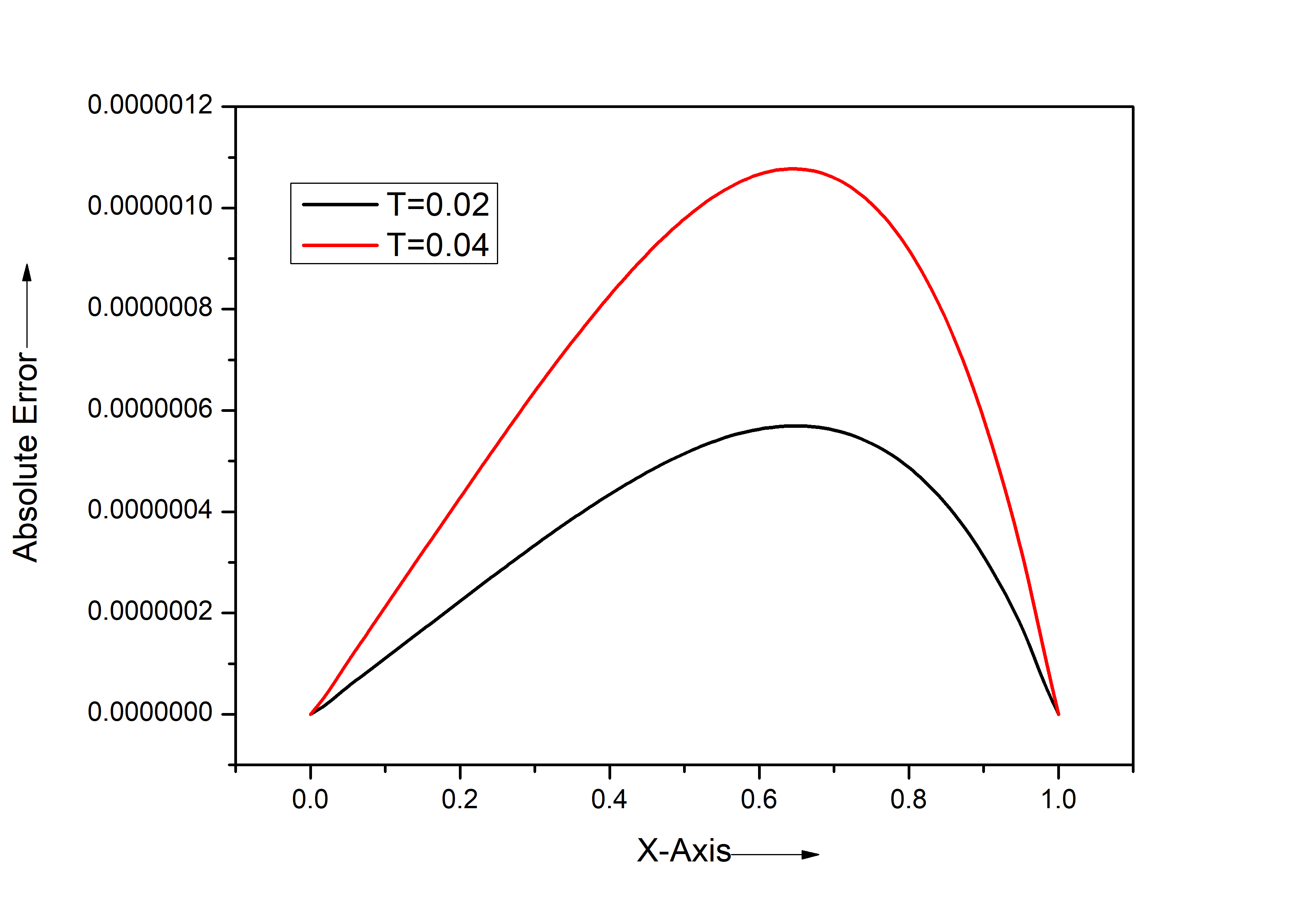}
	\caption{Numerical solution (left) and absolute error (right) for the problem \eqref{prob2} with $\Delta t=0.001,~\nu=1,~\sigma=1$ at different time $T$ for $J=4$.}
	\label{F2}
	\end{figure}
\begin{figure}[!ht]
	\includegraphics[height=2.5in,width=3.5in]{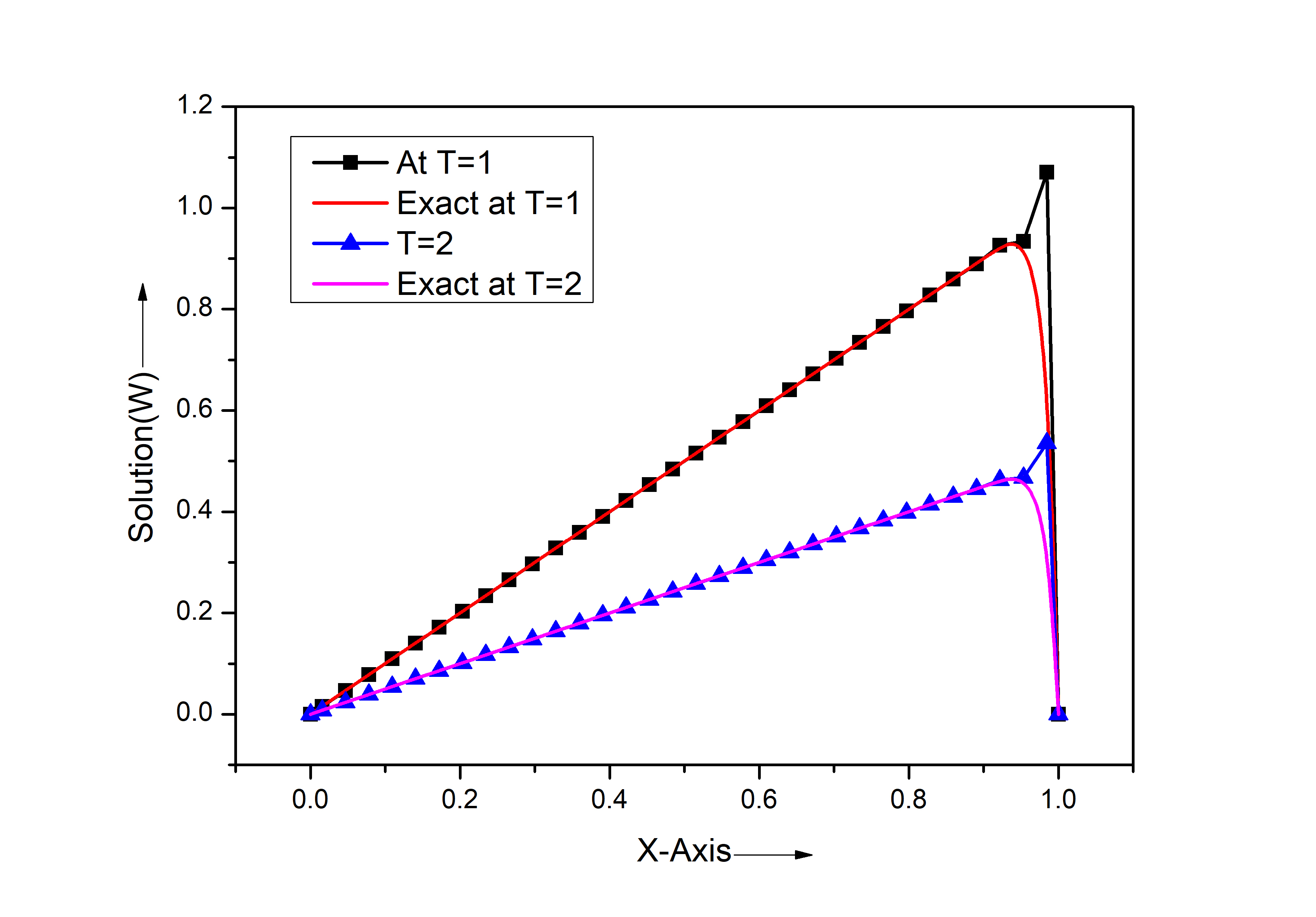}
	\includegraphics[height=2.5in,width=3.5in]{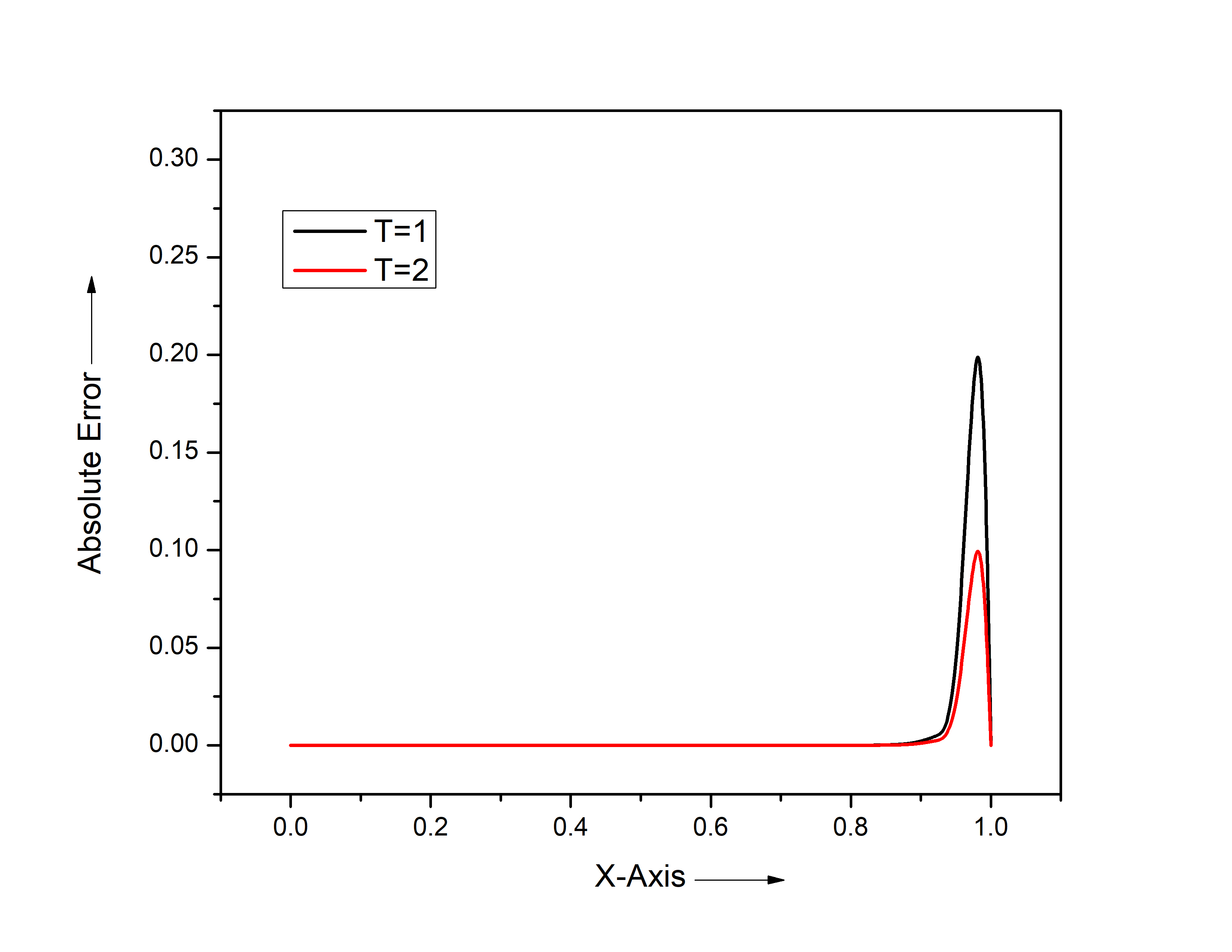}
	\caption{Numerical solution and exact solution (left) and absolute error (right) for the problem \eqref{prob2} with $\Delta t=0.01,~\nu=0.01,~\sigma=4$ at different time $T$ for $J=4$.}
	\label{F2.2}
\end{figure}
 \subsection{Test Problem 3}
 \label{prob3}
 Here we take $\mu=1,\delta=0$, then equation $\eqref{eq1}$ becomes one dimensional Burgers' equation. We take BCs
\begin{eqnarray}
w(0,t)=0=w(1,t),~~~t>0,
\end{eqnarray} 
and the initial conditions
\begin{eqnarray}
w(x_{*},0)=\frac{2\pi \nu  \sin(\pi x_{*})}{\sigma +\cos(\pi x_{*})}, ~x_{*}\in(0,1),
\end{eqnarray}
which is obtained from the exact solution \cite{asaithambi2010numerical}
\begin{eqnarray}
w(x_{*},t)=\frac{2\pi \nu e^{-\pi^2\nu t} \sin(\pi x_{*})}{\sigma +e^{-\pi^2\nu t}\cos(\pi x_{*})}, ~x_{*}\in(0,1),
\end{eqnarray}
where $\sigma >1$ is a parameter. \par  The numerical results of the example for $\nu=0.01$ are presented in table \ref{TN3} with the time step $\Delta t=0.01$ at $T=1$ for the parameter $\sigma=100$. Table summerize $L_{\infty}$ and $L_{2}$-error norm. The result by the present method is compared with the result published in \cite{rahman2010some} and \cite{mittal2012numerical} and is found that result is much better than the result in \cite{rahman2010some} and \cite{mittal2012numerical}. In the present method the spacial step size $\Delta x$ is greater than the spacial step size taken in \cite{rahman2010some} and \cite{mittal2012numerical} and found that the error is comparatively small. In fig. \ref{F3} (Left) numerical solution and exact solution is depicated for small value of $\nu=0.005$ at differnt $T$ with the time step $\Delta t=0.01$ and $\sigma =4$. It is observed that numerical result are very closed to the exact solution. In fig.\ref{F3} (Right) the absolute error is ploted for different $T$. It can been seen that the absolute error are very small and less than $0.0000010$ which is acceptable.
\begin{table}[!ht]
	\centering
	\caption{Comparision of numerical result with the existing result by the help of $L_{\infty}$ and $L_{2}$ error of the problem \eqref{prob3} at $T=1$ for $\nu=0.01,\Delta t=0.01$ and $\sigma=100.$}
	\label{TN3}
	\begin{tabular}{llllllll}
		\hline
		& \multicolumn{2}{l}{Kaysar\cite{rahman2010some}} & \multicolumn{2}{l}{Mittal and Jain \cite{mittal2012numerical}} &    & \multicolumn{2}{l}{Present} \\ 
	&	\multicolumn{2}{l}{\noindent\rule{4.0cm}{0.6pt}}        & \multicolumn{2}{l}{\noindent\rule{4.0cm}{0.6pt}} &\multicolumn{3}{l}{\noindent\rule{6.0cm}{0.6pt}}\\ 
		$\Delta x$ \hspace{0.2in} & $L_{2}-error$     & $L_{\infty}-error$ \hspace{0.4in}& $L_{2}-error$     & $L_{\infty}-error$ \hspace{0.4in}& $\Delta x$ & $L_{2}-error$      & $L_{\infty}-error$ \\  \hline
		1/10 & 3.4545E-07  & 4.8808E-07   & 3.2840E-07  & 4.6280E-07   &1/ 8  & 2.52147E-07  & 3.58275E-07  \\
		1/20 & 1.0124E-07  & 1.4305E-07   & 8.1921E-08  & 1.1640E-07   & 1/16 & 6.35077E-08  & 9.02969E-08  \\ 
		1/40 & 4.0028E-08  & 5.6677E-08   & 2.0470E-08  & 2.9068E-08   & 1/32 & 1.59079E-08  & 2.26455E-08  \\ 
		1/80 & 4.0028E-08  & 3.4992E-08   & 5.1194E-09  & 7.2706E-09   & 1/64 & 3.98117E-09  & 5.66586E-09  \\ \hline
	\end{tabular}
\end{table}
\begin{figure}[!ht]
	\includegraphics[height=2.5in,width=3.5in]{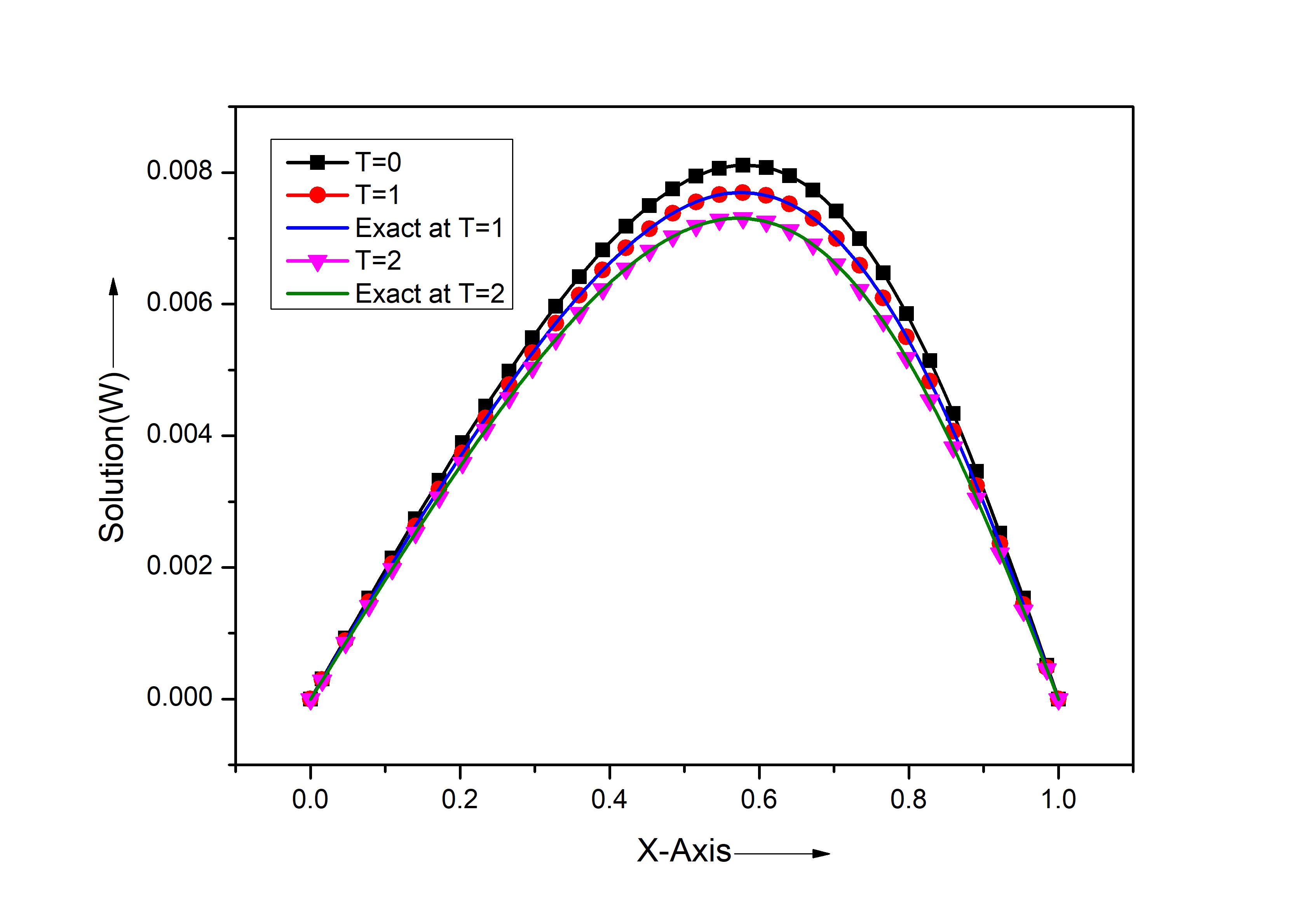}
	\includegraphics[height=2.5in,width=3.5in]{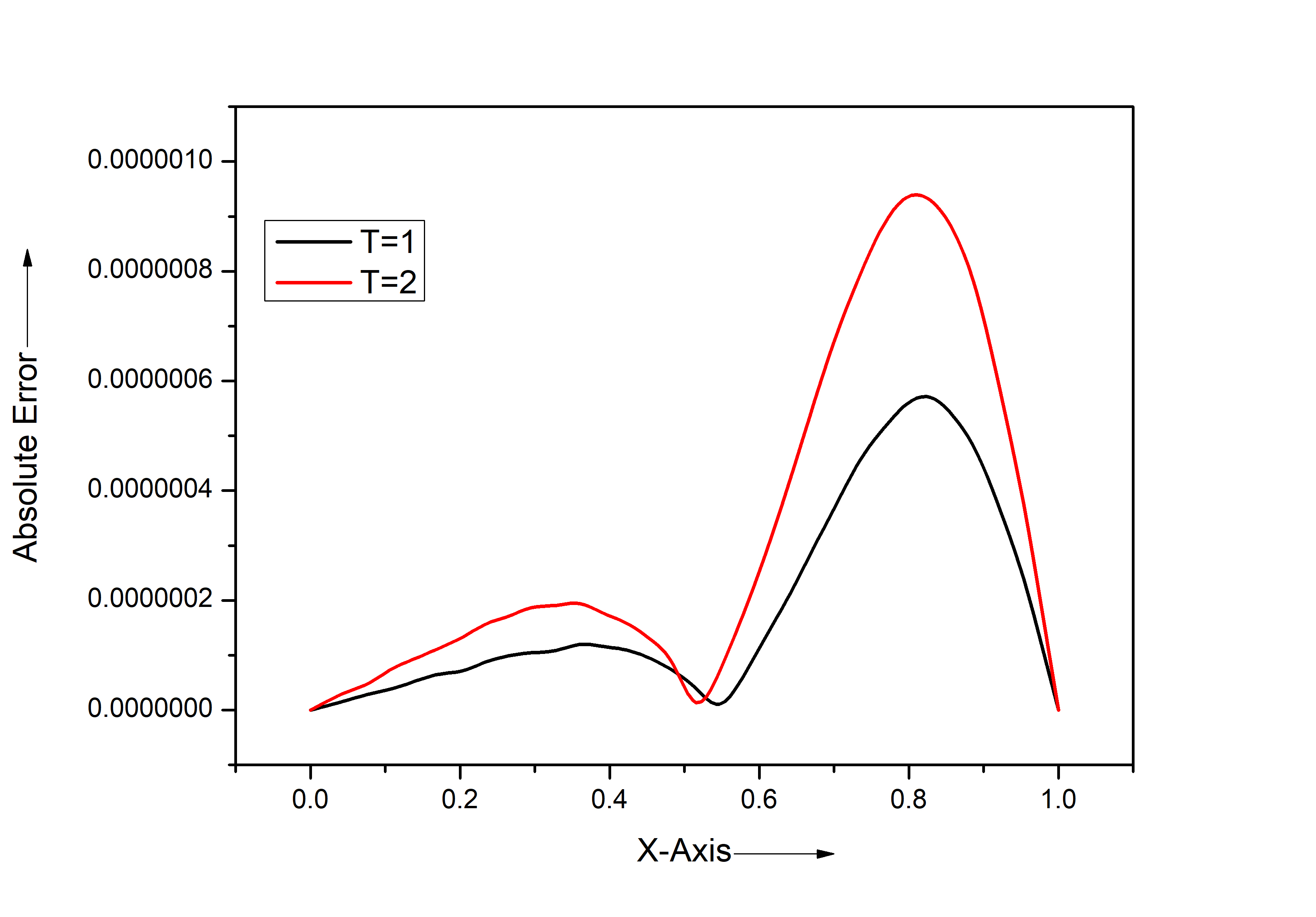}
	\caption{Numerical solution (left) and absolute error (right) for the problem \eqref{prob3}  with $\nu=0.005,~\Delta t=0.01,~\sigma=4$, at different time $T$ for $J=4$.}
	\label{F3}
\end{figure}
\subsection{Test Problem 4}
\label{prob4}
Here we take $\mu=2,\delta=1 $, then the equation \eqref{eq1} becomes 
\begin{eqnarray}
w_{t}+w^2w_{x_{*}}= \nu w w_{x_{*}^2}.
\end{eqnarray}
Let us take boundary condition 
\begin{eqnarray}
w(0,t)=0=w(5,t) ~~t>0,
\end{eqnarray}
with initial condition
\begin{eqnarray}
w(x_{*},0)=sin(\pi x_{*})
\end{eqnarray}
To the best of our knowledge the analytical solution of the problem \eqref{prob4} does not exist in the literature but the existence of solution is discussed in \cite{tersenov1999solvability}. We dont have exact solution so we only compute numerical result and plot in figure. In fig.\ref{F4} (Left) We plot  the numerical result of the problem \ref{prob4} for different values of $\nu $  at $T=0.1$ and time step $\Delta t=0.01$ and $2M=32$.  It is observed that the numerical solution of the problem follows the physical behaviour of the solution for all values of $\nu$. In fig.\ref{F4} (Right) numerical solution are ploted at different time $T$ for small values of $\nu=0.005$  and obsrved that it also follows physical behaviour of the solution.
\begin{figure}[!ht]
\includegraphics[height=2.5in,width=3.5in]{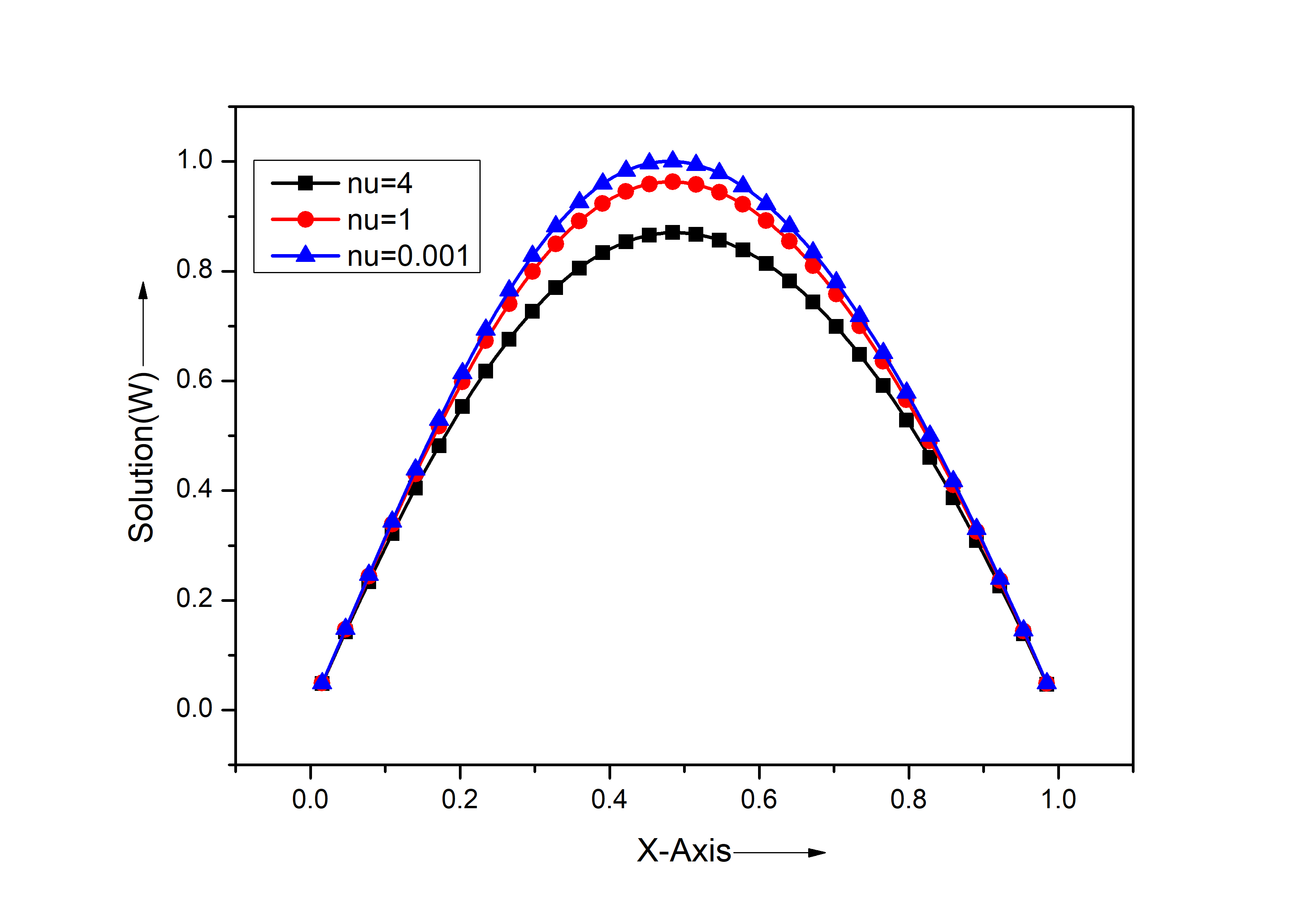}
\includegraphics[height=2.5in,width=3.5in]{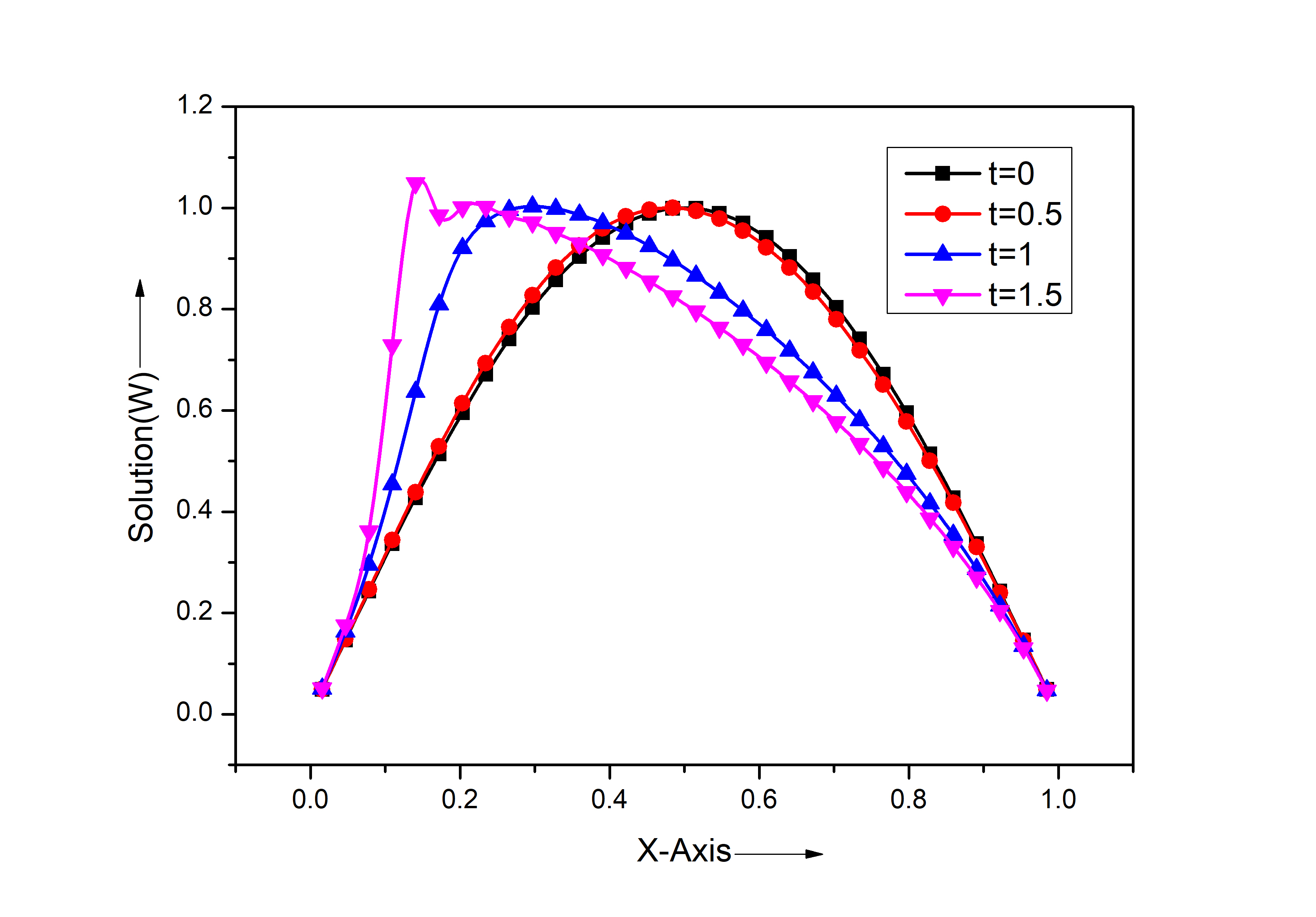}
	\caption{Numerical solution of the problem \eqref{prob4} at different $\nu $ and  $T=0.1 $ with $\Delta t=0.01$ (left)  and for the problem \eqref{prob4} at different time $T, \nu=0.005$ (right) for $J=4$.}
	\label{F4}
\end{figure}
\section{Conclusion}
\label{sec9}
In this work, Haar wavelet with the combination of quasilinearization and finite forward difference which involves averaging is discussed. The performance of the  present method is shown by testing the method over several examples and accuracy of the method is measured by $L_{2}$ and $L_{\infty}$-error norm. It is observed that the present method gives better accuracy than the result published in the literature even for small number of  grid points. Based on the performance of the present method, it is observed that the our method is competitive with the existing method such as finite difference, finite element etc., and this method can be also used for different type of PDEs that models real life problems in different field of engineering and science. 
\bibliographystyle{unsrt}
\bibliography{wavelet1}
\end{document}